\documentclass[12pt]{amsart}

\usepackage{
	amsmath,
 	amsfonts,
  	amssymb,
 	amsthm,
        datetime,
	}

\usepackage{tikz-cd}
\usepackage{tikz}
\usepackage{ytableau}

\usepackage{rotating}

\usetikzlibrary{decorations.pathmorphing, decorations.pathreplacing}
\usetikzlibrary{decorations.pathreplacing,backgrounds,decorations.markings}
\usetikzlibrary{arrows,shapes,positioning}

\usepackage{color}

\usepackage{stmaryrd}
\usepackage[cal=boondoxo]{mathalfa}
\usepackage{mathtools}
\usepackage{bbm}

\usepackage{hyperref}
\usepackage[capitalise]{cleveref}

\usepackage{a4wide}

\usepackage{mathabx}

\theoremstyle{plain}
\newtheorem{thm}{Theorem}[section]
\newtheorem{cor}[thm]{Corollary}

\newtheorem{prop}[thm]{Proposition}
\newtheorem{conj}[thm]{Conjecture}

\theoremstyle{definition}
\newtheorem{rem}[thm]{Remark}

\newtheorem{ex}[thm]{Example}

\theoremstyle{definition}
\newtheorem{defn}[thm]{Definition}

\DeclareMathOperator{\Symb}{Symb}

\def\makeautorefname#1#2{\expandafter\def\csname#1autorefname\endcsname{#2}}
\makeautorefname{lem}{Lemma}%
\makeautorefname{prop}{Proposition}%
\makeautorefname{rem}{Remark}%
\makeautorefname{section}{Section}%

\DeclareMathOperator{\id}{Id}

\DeclareMathOperator{\Ind}{Ind}
\DeclareMathOperator{\Res}{Res}

\DeclareMathOperator{\im}{Im}

\DeclareMathOperator{\GL}{GL}
\DeclareMathOperator{\refl}{Ref}
\DeclareMathOperator{\Gau}{Gau}
\DeclareMathOperator{\Irr}{Irr}

\DeclareMathOperator{\JM}{JM}

\DeclarePairedDelimiter\abs{\lvert}{\rvert}
\DeclarePairedDelimiter\ket{\lvert}{\rangle}

\definecolor{shadecolor}{gray}{0.90}
\def\boitegrise#1#2{\begin{centerline}{\fcolorbox{black}{shadecolor}{~
    \begin{minipage}[t]{#2}{\vphantom{~}#1\vphantom{$A_{\displaystyle{A_A}}$}}
            \end{minipage}~}}\end{centerline}\medskip}

\def\arxivprefix{https://arxiv.org}

\definecolor{linkcolor}{rgb}{0,0,1} 
\hypersetup{
	colorlinks=true, 
	pdfstartview=FitV, 
	urlcolor=linkcolor, 
	linkcolor= linkcolor, 
	citecolor=linkcolor 
}

\title[On Calogero--Moser cellular characters for $G(l,1,n)$]{On Calogero--Moser cellular characters for imprimitive complex reflection groups}

\author{Nicolas Jacon}
\address{UFR Sciences exactes et naturelles, 
Laboratoire de Math\'ematiques UMR CNRS 9008, Moulin de la Housse BP 1039
51100 Reims, France.}
\email{nicolas.jacon@univ-reims.fr}
\author{Abel Lacabanne}
\address{Laboratoire de Math\'ematiques Blaise Pascal (UMR 6620), Universit\'e Clermont Auvergne, Campus Universitaire des C\'ezeaux, 3 place Vasarely, 63178 Aubi\`ere Cedex, France}
\email{abel.lacabanne@uca.fr}

\begin{document}


\begin{abstract}
We study the relationship between Calogero--Moser cellular characters and characters defined from vectors of a Fock space of type $A_{\infty}$. Using this interpretation, we show that Lusztig's constructible characters of the Weyl group of type $B$ are sums of Calogero--Moser cellular characters. We also give an explicit construction of the character of minimal $b$-invariant of a given Calogero--Moser family of the complex reflection group $G(l,1,n)$. 
\end{abstract}


\maketitle



\section{Introduction}
\label{sec:intro}

The representation theory of finite Coxeter groups, and more precisely the Kazhdan--Lusztig theory, is of crucial importance for the understanding of various problems in Lie theory. For example, in his classification of unipotent characters of finite groups of Lie type, Lusztig has underlined the importance of a partition of irreducible characters of the associated Weyl group into families. The starting point of Kazhdan--Lusztig theory is the construction of a basis of Hecke algebras in \cite{kazhdan-lusztig}, known nowadays as the Kazhdan--Lusztig basis. From this basis, one can define partitions of the Coxeter group in the so-called right, left or two-sided cells. To every right cell is attached a module over the Hecke algebra and a character (not necessarily irreducible) of the Coxeter group. These characters are the cellular characters\footnote{The associated modules have not to be confused with the cell modules in the sense of Graham and Lehrer, even if there are closed relationships between these two objects \cite{geck-cellular}}.  Lusztig has then realized that these constructions are also possible with Hecke algebras at unequal parameters \cite{lusztig-unequal}. Many questions around Kazhdan--Lusztig theory for Hecke algebra with unequal parameters still remain open, see \cite{bonnafe-unequal} for a recent monograph about this subject.

Since the Coxeter groups are real reflection groups, one may ask the following question: does there exist a Kazhdan--Lusztig theory for finite complex reflection groups? Of course, one first need the notion of Hecke algebras for complex reflection groups. These objects have been introduced by Brou\'e--Malle--Rouquier in \cite{broue-malle-rouquier}, where presentations in terms of braid relations are also given. In the particular case of the complex reflection group $G(l,1,n)$, they recover an algebra already introduced in \cite{ariki-koike}, nowadays known as the Ariki--Koike algebras. Unfortunately, we do not know a basis of the Hecke algebra associated with a finite complex reflection group which has similar properties as the Kazhdan--Lusztig basis for Hecke algebras associated with Coxeter groups.

But this story does not stop there. Rouquier \cite{rouquier-blocs} managed to give an interpretation of the Kazhdan--Lusztig families of cristallographic Coxeter groups in terms of blocks of the Hecke algebra, which extends to the complex reflection group case. These Rouquier families have been computed for all finite complex reflection groups \cite{chlouveraki}. The next major step in a tentative of the generalization of the Kazhdan--Lusztig theory has been made by Gordon, via the study of the representation theory of restricted rational Cherednik algebras at $t=0$ \cite{gordon-baby}. He introduced the notion of a  Calogero--Moser family, which is a partition of the irreducible characters of a given complex reflection group. If the complex reflection group is a Coxeter group, it is worth to notice that the construction of the Calogero--Moser families and the Kazhdan--Lusztig families are very different. But it has been shown, via explicit computations on both sides, that these two notions coincide for most finite irreducible Coxeter groups \cite[Theorem 2.4]{bellamy-thiel}.

Finally, Bonnaf\'e and Rouquier \cite{bonnafe-rouquier} defined the notion of Calogero--Moser right, left and two-sided Calogero--Moser cells by studying the geometry of the Calogero--Moser space associated with a complex reflection group. Their approach is partially motivated by the ideas of Gordon in which  the rational Cherednik algebra at $t=0$ plays an essential role. To each right cell is also associated a Calogero--Moser cellular character. If the complex reflection group is a Coxeter group, they conjecture that the various output of this Calogero--Moser theory is the same as the output of the Kazhdan--Lusztig theory (families, cells, cellular characters). However, only the some particular cases are well understood: the symmetric group \cite{brochier-gordon-white}, partially the dihedral groups \cite{bonnafe-cellular-dihedral,bonnafe-germoni} or some exemples in small ranks \cite[Chapter 18 and 19]{bonnafe-rouquier}.

In this paper, we are mainly interested in the case of the complex reflection group $G(l,1,n)$ and more specifically in the cellular characters for $G(l,1,n)$. The second author has stated in \cite{lacabanne} some conjectures relating these Calogero--Moser cellular characters with the expressions of the canonical basis of the level $l$ Fock space for $U_q(\mathfrak{gl}_{\infty})$. Thanks to this interpretation of constructible characters for the Weyl group of type $B_n$ settled by Leclerc and Miyachi \cite{leclerc-miyachi}, one of our main result is the following, see \cref{thm:constructible_CM} for the precise details on the parameters involved.

\begin{thm}
  Lusztig's constructible characters for the Weyl group of type $B_n$ are sums of Calogero--Moser cellular characters.
\end{thm}

Our second main result concerns the computation of certain remarkable elements in the Calogero--Moser families. We  provide an explicit proof of a result concerning the $b$-invariants of the characters of a Calogero--Moser family. Namely, it has been proven that for every such family, there exists a unique irreducible character of minimal invariant among the family \cite[Theorem 7.4.1(b)]{bonnafe-rouquier}. We here give another proof of this result, in the specific case of $G(l,1,n)$, and explicitly provide a description in terms of $l$-symbols of the character of minimal $b$-invariant in a Calogero--Moser family. We refer to \cref{sec:b-invariant} for more details, and especially to \cref{thm:minimal_binv}.

\medskip

The organisation of the paper is as follows. In \cref{sec:cm-characters}, we review the definition of Calogero--Moser cellular characters via Gaudin algebras, as in \cite[Chapter 8]{bonnafe-rouquier}, and focus then in the specific case of the complex reflection group $G(l,1,n)$. \cref{sec:monomial-bases} concerns the Fock spaces, and  the quasimonomial vectors, that we relate to some approximations of the Calogero--Moser cellular characters. Then, in \cref{sec:typeB}, we study in details the case of the Weyl group of type $B$ and the relationship between Calogero--Moser cellular characters and Lusztig's constructible characters.  Before the last section, we go back to the case of the complex reflection group $G(l,1,n)$ and describe the symbol of minimal $b$-invariant in a Calogero--Moser family using the combinatorics of $l$-symbols. Finally, we end in \cref{sec:conjectures} with some various conjectures, mostly related with the $b$-invariant.\\
\\
{\bf Acknowledgements.} A.L. is supported by a PEPS JCJC grant from INSMI (CNRS), N.J. is supported by ANR AHA 18-CE40-0001  and ANR CORTIPOM  21-CE40-001. The authors would like to thank C\'edric Bonnaf\'e, Maria Chlouveraki, J\'er\'emie Guilhot and  Dario Mathia for useful discussions.



\section{Calogero--Moser cellular characters}
\label{sec:cm-characters}

The Calogero--Moser cellular characters are a conjectural extension to complex reflection groups of cellular characters in Kazhdan--Lusztig theory. The definition of cells usually requires the study of representation theory of Cherednik algebras, but we will use the definition of the Calogero--Moser characters using the notion of Gaudin algebras. We recall now their definition and then focus on the case of $G(l,1,n)$.

\subsection{Set-up}
\label{sec:set-up}

Let $V$ be a finite dimensional $\mathbb{C}$-vector space with dual $V^*$. The duality pairing is denoted by $\langle\cdot,\cdot\rangle\colon\! V\times V^* \rightarrow \mathbb{C}$ and the determinant map by $\det \colon\! \GL(V)\rightarrow \mathbb{C}^*$.

Let $W \subset \GL(V)$ be a complex reflection group. We denote by $\refl(W)$ the set of pseudo-reflections of $W$, that is elements of $W$ which fix an hyperplane. For each $s\in \refl(W)$, we choose $\alpha_s \in V^*$ and $\alpha_s^\vee\in V$ such that
\[
  \ker(s-\id) = \ker(\alpha_s)\quad\text{and}\quad\im(s-\id) = \mathbb{C}\alpha_s^\vee.
\]
In other terms, we have $s(v) = v - (1-\det(s))\frac{\langle v,\alpha_s\rangle}{\langle \alpha_s^\vee,\alpha_s\rangle}\alpha_s^\vee$ for any $v\in V$.

We denote by $\Irr(W)$ the set of irreducible complex characters of $W$ and for any $\chi \in \Irr(W)$, we fix $V_\chi$ a simple representation of $W$ affording the character $\chi$.

Finally, we also fix a function $c\colon\! \refl(W) \rightarrow \mathbb{C},\ s\mapsto c_s$ which is invariant by conjugation of $W$ on $\refl(W)$.

\subsection{Definition via Gaudin algebras}
\label{sec:def-gaudin}

For any $x \in V$, we consider the following element of $\mathbb{C}(V)[W]$
\[
  \mathcal{D}_x = \sum_{s\in \refl(W)}c_s\det(s)\frac{\langle x, \alpha_s\rangle}{\alpha_s}s.
\]

\begin{defn}
  The Gaudin algebra with parameter $c$, denoted by $\Gau_c(W)$ is the $\mathbb{C}(V)$-subalgebra of $\mathbb{C}(V)[W]$ generated by $(\mathcal{D}_x)_{x\in V}$.
\end{defn}

Since $x\mapsto \mathcal{D}_x$ is clearly linear, the Gaudin algebra $\Gau_c(W)$ is generated by $(\mathcal{D}_x)_{x \in B}$, where $B$ is a basis of $V$.

\begin{prop}[{\cite[Proposition 8.3.1]{bonnafe-rouquier}}]
  For any $x,y \in V$, the elements $\mathcal{D}_x$ and $\mathcal{D}_y$ commute in $\mathbb{C}(V)[W]$. Therefore $\Gau_c(W)$ is a commutative subalgebra of $\mathbb{C}(V)[W]$.
\end{prop}

Even though $\Gau_c(W)$ is commutative, its representation theory is far form being well-understood: it depends heavily on $c$ and the field $\mathbb{C}(V)$ is almost never a splitting field. Some simple representations are not absolutely simple and therefore a simple representation of $\Gau_c(W)$ is not automatically of dimension $1$.

\begin{defn}
  \label{def:cm-char}
  Let $L$ be a simple representation of $\Gau_c(W)$. The Calogero--Moser cellular character associated with $L$ is
  \[
    \gamma_L^{\mathrm{CM}} = \sum_{\chi\in \Irr(W)}\left[\Res_{\Gau_c(W)}^{\mathbb{C}(V)[W]}(\mathbb{C}(V)\otimes_{\mathbb{C}}V_\chi)\colon L\right]\chi.
  \]
\end{defn}

The name Calogero--Moser cellular characters comes from another definition of these characters, which is related with the Calogero--Moser partition of $W$ into cells, see \cite{bonnafe-rouquier}. 

\subsection{The case of $G(l,1,n)$ and Jucys--Murphy elements}
\label{sec:gl1n}

We now focus on the case of the complex reflection group $G(l,1,n)$. We fix $\zeta=\exp(2i\pi/l)$, $\mu_l=\langle\zeta\rangle$ the group of $l$-th root of unity in $\mathbb{C}^*$ and suppose that $V$ is of dimension $n$ with basis $(y_i)_{1 \leq i \leq n}$ and dual basis $(x_i)_{1 \leq i \leq n}$. Identifying $\GL(V)$ with $\GL_n(\mathbb{C})$ through the basis $(y_i)_{1\leq i \leq n}$, the group $G(l,1,n)$ consist of monomial matrices with entries in $\mu_l$. By convention, $G(l,1,0)$ will denote the trivial group.

We denote by $s_{i,j}$ the reflection exchanging $y_i$ and $y_j$ and fixing $y_k$ for $k\not\in\{i,j\}$ and by $\sigma_i$ the reflection fixing $y_k$ for $k\neq i$ and sending $y_i$ to $\zeta y_i$. The group $G(l,1,n)$ is then generated by $\left\{s_{i,i+1}\ \middle\vert\ 1 \leq i < n\right\}\cup\{\sigma_1\}$. The reflections of $G(l,1,n)$ fall into $l$ conjugacy classes, namely $S_0 = \left\{\sigma_j^rs_{i,j}\sigma_j^{-r}\ \middle\vert 1\leq i\neq j \leq n, 0 \leq r < l\right\}$ and $S_1,\ldots,S_{l-1}$ given by $S_k=\left\{\sigma_i^k\ \middle\vert 1 \leq i \leq n\right\}$. Finally, we set $c_i=c_{\vert S_i}$ as a shorthand for the parameter $c$. Another set of parameters appears also in the context of Cherednik and Gaudin algebras, namely $k_0,\ldots,k_{l-1}$ given by
\[
  k_i = \frac{1}{l}\sum_{j=1}^{l-1}\zeta^{j(1-i)}c_j.
\]
We will consider the indices modulo $l$ and note that $k_0+\cdots+k_{l-1}=0$.

The following Jucys--Murphy elements have been introduced by the second author in \cite{lacabanne}
\[
  J_i = \sum_{\substack{s \in \refl(G(l,1,i))\\s \not\in \refl(G(l,1,i-1))}}c_s\det(s)s.
\]
These elements generate a commutative subalgebra $\JM_c(l,n)$ of $\mathbb{C}[G(l,1,n)]$. Similarly to \cref{def:cm-char}, we define the Jucys--Murphy cellular characters.

\begin{defn}
    \label{def:jm-char}
  Let $L$ be an irreducible representation of $\JM_c(l,n)$. The Jucys--Murphy cellular character associated with $L$ is
  \[
    \gamma_L^{\mathrm{JM}} = \sum_{\chi\in \Irr(W)}\left[\Res_{\JM_c(l,n)}^{\mathbb{C}[G(l,1,n)]}(V_\chi)\colon L\right]\chi.
  \]
\end{defn}

The Jucys--Murphy cellular characters approximate the Calogero--Moser cellular characters in the following sense:

\begin{prop}[{\cite[Theorem 1.10]{lacabanne}}]
  \label{prop:JM_sums_of_CM}
  Every Jucys--Murphy cellular character is a sum of Calogero--Moser cellular characters.
\end{prop}

We now aim to give an inductive characterization of these Jucys--Murphy cellular characters. We first recall the combinatorics of the $l$-partitions of $n$ which govern the representation theory of $G(l,1,n)$, see \cite[Section 5.1]{geck-jacon}.

A partition $\lambda$ of $n$ is a decreasing sequence of non-negative integers $\lambda_1\geq \lambda_2\geq\cdots\geq \lambda_k \geq \cdots$ such that $\lambda_i=0$ for $i\gg 0$ and $\abs{\lambda}:=\sum_{i=1}^\infty\lambda_i=n$. A $l$-partition of $n$ is then a $l$-tuple of partitions $(\lambda^{(1)},\ldots,\lambda^{(l)})$ such that $\sum_{i=1}^l\abs{\lambda^{(i)}} = n$.

The isomorphism classes of irreducible representations of $G(l,1,n)$ are parametrized by the set of $l$-partitions of $n$. For every $l$-partition $\lambda$ of $n$, we denote by $V_\lambda$ the corresponding irreducible representation and by $\chi_\lambda$ its character.

Restriction and induction can be described using by the Young graph of $l$-partitions of $n$. Given a $l$-partition $\lambda$, its Young diagram $[\lambda]$ is the set
\[
  \left\{(a,b,c)\in \mathbb{Z}_{>0}\times\mathbb{Z}_{>0}\times\{1,\ldots,l\}\ \middle\vert \ 1 \leq b \leq \lambda_a^{(c)}\right\}
\]
whose elements are called boxes. A box $\gamma$ of the Young diagram of a $l$-partition $\lambda$ of $n$ is said to be removable if $[\lambda]\backslash \{\gamma\}$ is the Young diagram of a $l$-partition $\mu$ of $n-1$. The box $\gamma$ is also said to be addable to the Young diagram of $\mu$.

\begin{prop}[{\cite[Proposition 5.1.8]{geck-jacon}}]
  Let $\lambda$ be a $l$-partition of $n$. Then
  \[
    \Res_{G(l,1,n-1)}^{G(l,1,n)}(V_\lambda) = \oplus_{\mu}V_\mu
  \]
  where the sum is taken over the $l$-partitions $\mu$ of $n-1$ whose Young diagram is obtained by removing a box of $[\lambda]$ and
  \[
    \Ind_{G(l,1,n)}^{G(l,1,n+1)}(V_\lambda) = \oplus_{\mu}V_\mu
  \]
  where the sum is taken over the $l$-partitions $\mu$ of $n+1$ whose Young diagram is obtained by adding a box to $[\lambda]$.
\end{prop}

Using these decompositions, we obtain a basis of $V_\lambda$ indexed by the so-called standard tableaux of shape $\lambda$. Given a $l$-partition $\lambda$ of $n$, a standard tableau of shape $\lambda$ is a bijection $\mathfrak{t}\colon\! [\lambda]\rightarrow \{1,\ldots,n\}$ such that for all boxes $\gamma=(a,b,c)$ and $\gamma'=(a',b',c)$ of $\lambda$, we have $\mathfrak{t}(\gamma) < \mathfrak{t}(\gamma')$ if $a=a'$ and $b<b'$ or $a<a'$ and $b=b'$. The datum of a standard tableau of shape $\lambda$ is equivalent to the datum of a finite sequence $(\lambda[i])_{1\leq i \leq n}$ such that $\lambda[i]$ is a $l$-partition of $i$, $\lambda[n] = \lambda$ and $[\lambda[i+1]]$ is obtained from $[\lambda[i]]$ by adding a box. We denote by $\lambda^{\mathfrak{t}}$ the sequence of $l$-partitions obtained from the standard tableau ${\mathfrak{t}}$, that is $[\lambda^{\mathfrak{t}}[i]] = {\mathfrak{t}}^{-1}\{1,\ldots,i\}$.

Therefore, using the branching rule for restriction, we have $V_\lambda = \bigoplus_{\mathfrak{t}}D_{\mathfrak{t}}$ where the sum is over the standard tableaux of shape $\lambda$ and $D_{\mathfrak{t}}$ is a one dimensional space contained in the irreducible component $V_{\lambda^{\mathfrak{t}}[i]}$ of $\Res_{G(l,1,i)}^{G(l,1,n)}(V_\lambda)$ for all $1 \leq i \leq n$.

The Jucys--Murphy elements act diagonally on $V_\lambda$ with respect to the decomposition $V_\lambda = \bigoplus_{\mathfrak{t}}D_{\mathfrak{t}}$. On $D_{\mathfrak{t}}$, the Jucys--Murphy element $J_i$ acts by multiplication by $l(k_{1-c}-c_0(b-a))$, where $\mathfrak{t}(a,b,c) = i$, see \cite[Corollary 1.8]{lacabanne}.

Given $\xi\in \mathbb{C}$ and $\lambda$ a $l$-partition of $n-1$, we refine the induction as follows:
\[
  \xi-\Ind_{G(l,1,n-1)}^{G(l,1,n)}(V_\lambda) = \bigoplus_{\mu}V_\mu,
\]
the sum being over the $l$-partitions $\mu$ of $n$ whose Young diagram is obtained by adding a box to $[\lambda]$ and such that $J_{n}$ acts on $V_\mu$ by multiplication by $\xi$. Using this truncated induction, we recursively define a set of representations of $G(l,1,n)$ as follows:
\[
  \mathcal{E}_0 = \{\mathbf{1}\}\quad\text{and}\quad
  \mathcal{E}_n =\left\{\xi-\Ind_{G(l,1,n-1)}^{G(l,1,n)}(V)\ \middle\vert\ V \in \mathcal{E}_{n-1},\ \xi\in \mathbb{C}\right\}\backslash\{0\},
\]
where $\mathbf{1}$ denotes the trivial representation of the trivial group. From the above discussion, we immediately deduce the following new way to define the Jucys-Murphy cellular characters. 

\begin{prop}
  \label{prop:JM-char-truncated-induction}
  The set of characters of representations in $\mathcal{E}_n$ coincides with the set of Jucys--Murphy constructible characters.
\end{prop}


\begin{rem}
  Instead of considering the Jucys-Murphy elements $J_i$, we might consider the shifts $J_i+k$ by any scalar $k$. The set $\mathcal{E}_n$ does not change if we replace the Jucys-Murphy elements by these shifts.

  In terms of parameters, this is equivalent to replace the condition $k_0+\cdots+ k_{l-1}=0$ by $k_0+\cdots+k_{l-1}=k$. Therefore, we may and will remove the condition $k_0+\cdots +k_{l-1}=0$ on the parameters without changing the set of Jucys-Murphy constructible characters.
\end{rem}



\section{Quasimonomial vectors of Fock spaces in type $A_{\infty}$}
\label{sec:monomial-bases}
%

Following \cite[Section 6.2]{geck-jacon}, we introduce the Fock space as a representation of $U_q(\mathfrak{sl}_{\infty})$ together with quasimonomial vectors. We then relate evaluations at $q=1$ of these vectors with the Jucys--Murphy cellular characters introduced in the previous section.

We fix $q$ an indeterminate and define the usual quantum integers and factorials in $\mathbb{Z}[q,q^{-1}]$ by
\[
  [n] = \frac{q^n-q^{-n}}{q-q^{-1}}\quad\text{and}\quad [n]! = \prod_{k=1}^{n}[k],
\]
for $n\in \mathbb{N}$.

\subsection{The algebra $U_q(\mathfrak{sl}_{\infty})$ and the Fock space $\mathcal{F}_{\mathbf{s}}$}
\label{sec:algebra-sl}

Let $U_q(\mathfrak{sl}_{\infty})$ be the $\mathbb{Q}(q)$-algebra with generators $E_i,F_i$ and $K_i^{\pm 1}$ for $i \in \mathbb{Z}$ subject to
\begin{gather*}
  K_iK_i^{-1} = 1 = K_i^{-1}K_i,\quad K_i K_j = K_j K_i, \quad K_iE_j = q^{-\delta_{i-1,j}+2\delta_{i,j}-\delta_{i+1,j}}E_jK_i,\\
  K_iF_j = q^{\delta_{i-1,j}-2\delta_{i,j}+\delta_{i+1,j}}F_jK_i, \quad [E_i,F_j] = \delta_{i,j}\frac{K_i-K_i^{-1}}{q-q^{-1}},
\end{gather*}
and the Serre relations
\begin{align*}
  E_i^2E_j - [2] E_iE_jE_i + E_jE_i^2 &= 0, & F_i^2F_j - [2] F_iF_jF_i + F_jF_i^2 &= 0, &\text{ if }\abs{i-j} = 1,\\
  [E_i,E_j] &= 0, & [F_i,F_j]&=0, &\text{ if }\abs{i-j} > 1.
\end{align*}

We define the divided powers $E_i^{(r)}$ and $F_i^{(r)}$ by
\[
  E_i^{(r)}=\frac{E_i}{[r]!}\quad\text{and}\quad F_i^{(r)}=\frac{F_i}{[r]!}.
\]
We endow $U_q(\mathfrak{sl}_{\infty})$ with a structure of Hopf algebra. The coproduct $\Delta$, the antipode $S$ and the counit $\varepsilon$  are given on the generators by
\begin{align*}
  \Delta(E_i) &=  E_i\otimes 1 + K_i^{-1}\otimes E_i, & S(E_i) &= -E_iK_i, & \varepsilon(E_i) &= 0,\\
  \Delta(F_i) &= F_i\otimes K_i + 1 \otimes F_i, & S(F_i) &= -K_i^{-1}F_i, & \varepsilon(F_i) &= 0,\\
  \Delta(K_i) &= K_i\otimes K_i, & S(K_i) &= K_i^{-1}, & \varepsilon(K_i) &= 1.
\end{align*}

The fundamental weights of $U_q(\mathfrak{sl}_{\infty})$ are denoted by $(\Lambda_k)_{k\in\mathbb{Z}}$. We now fix $l\geq 0$ and $\mathbf{s}=(s_1,\ldots,s_l)\in\mathbb{Z}^l$. We define the Fock space $\mathcal{F}_{\mathbf{s}}$ of charge $\mathbf{s}$ as the $\mathbb{Q}(q)$ vector space with basis $(\ket{\lambda,\mathbf{s}})$ where $\lambda$ runs in the set of $l$-partitions of integers. We set
\[
  \mathcal{F}^n_{\mathbf{s}} = \bigoplus_{\lambda}\mathbb{Q}(q)\ket{\lambda,\mathbf{s}}
\]
where $\lambda$ runs in the set of $l$-partitions of $n$ so that $\mathcal{F}_{\mathbf{s}}=\bigoplus_{n\in\mathbb{N}}\mathcal{F}^n_{\mathbf{s}}$.

In order to define an action of $U_q(\mathfrak{sl}_{\infty})$ on the Fock space $\mathcal{F}_{\mathbf{s}}$, we first need the notion of charged content of a box. Given $\lambda$ a $l$-partition and $\gamma=(a,b,c)$ a box of its Young diagram, the charged content (relative to $\mathbf{s}$) is the integer $b-a+s_c$. If $\gamma=(a,b,c)$ and $\gamma'=(a',b',c')$ are two removable or addable boxes of $\lambda$ with same content, we set $\gamma \prec_{\mathbf{s}} \gamma'$ (resp. $\gamma \succ_{\mathbf{s}} \gamma'$) if $c<c'$ (resp. $c>c'$). 

\begin{rem}\label{diffcon}
Note that two removable (resp. addable) boxes of same content must lie in different components of $[\lambda]$. Thus, for each $i\in \mathbb{Z}$, we have at most $l$ removable (resp. addable) boxes with content $i$ in $[\lambda]$. 
\end{rem}


Finally, if $\lambda$ and $\mu$ are two $l$-partitions respectively of $n$ and $n+1$ with $[\mu]=[\lambda]\cup\{\gamma\}$ and $\gamma$ of charged content $i$, we set
\begin{align*}
  N_i^{\succ}(\lambda,\mu)=&\abs*{\left\{\text{addable nodes }\gamma'\text{ to }\lambda\text{ of charged content }i\text{ with }\gamma'\succ_{\mathbf{s}}\gamma\right\}}\\
                          &-\abs*{\left\{\text{removable nodes }\gamma'\text{ of }\lambda\text{ of charged content }i\text{ with }\gamma'\succ_{\mathbf{s}}\gamma\right\}}\\
  N_i^{\prec}(\lambda,\mu)=&\abs*{\left\{\text{addable nodes }\gamma'\text{ to }\lambda\text{ of charged content }i\text{ with }\gamma'\prec_{\mathbf{s}}\gamma\right\}}\\
             &-\abs*{\left\{\text{removable nodes }\gamma'\text{ of }\lambda\text{ of charged content }i\text{ with }\gamma'\prec_{\mathbf{s}}\gamma\right\}}\\
  N_i(\lambda)=&\abs*{\left\{\text{addable nodes to }\lambda\text{ of charged content }i\right\}}\\
                           &-\abs*{\left\{\text{removable nodes of }\lambda\text{ of charged content }i\right\}}
\end{align*}

\begin{prop}[{\cite[Proposition 3.5]{jimbo-misra-miwa-okado}}]
  \label{prop:action-fock}
  The Fock space $\mathcal{F}_{\mathbf{s}}$ is an integrable $U_q(\mathfrak{sl}_{\infty})$-module with action given by
  \begin{align*}
    E_i\cdot \ket{\lambda,\mathbf{s}} &= \sum_{\substack{\mu,[\lambda]\backslash [\mu] = \{\gamma\}\\ \gamma\text{ of charged content }i}}q^{-N_i^{\prec}(\mu,\lambda)}\ket{\mu,\mathbf{s}}\\
    F_i\cdot \ket{\lambda,\mathbf{s}} &= \sum_{\substack{\mu,[\mu]\backslash [\lambda] = \{\gamma\}\\ \gamma\text{ of charged content }i}}q^{N_i^{\succ}(\lambda,\mu)}\ket{\mu,\mathbf{s}}\\
    K_i \cdot \ket{\lambda,\mathbf{s}} &= q^{N_i(\lambda)}\ket{\lambda,\mathbf{s}}.
  \end{align*}
\end{prop}

Therefore, the vector $\ket{\emptyset,\mathbf{s}}$ is a highest weight vector of weight $\sum_{i=1}^l\Lambda_{s_i}$.

A non-zero vector $v$ of $\mathcal{F}_{\mathbf{s}}^n$ is said to be \emph{quasimonomial} if it is obtained from $\ket{\emptyset,\mathbf{s}}$ by successive applications of $F_i$'s, that is if there exists a sequence of integers $(i_1,\ldots,i_n)$ such that
\[
  v = (F_{i_n}\cdots F_{i_1})\cdot \ket{\emptyset,\mathbf{s}}. 
\]

A non-zero vector $v$ of $\mathcal{F}_{\mathbf{s}}^n$ is said to be \emph{monomial} if it is obtained from $\ket{\emptyset,\mathbf{s}}$ by successive applications of divided powers of the $F_i$'s, that is if there exists a sequence of integers $(i_1,\ldots,i_k)$ with $i_{j+1}\neq i_j$ and $(r_1,\ldots,r_k)$ integers with $r_1+\cdots+r_k=n$ such that
\[
  v = (F^{(r_k)}_{i_k}\cdots F^{(r_1)}_{i_1})\cdot \ket{\emptyset,\mathbf{s}}. 
\]

From the formulas of the action in \cref{prop:action-fock} we easily see that monomial and quasimonomial vectors are elements of $\bigoplus \mathbb{N}[q,q^{-1}]\ket{\lambda,\mathbf{s}}$. 

To a quasimonomial vector $v\in\mathcal{F}_{\mathbf{s}}^n$, we associate a character $\gamma_v$ of $G(l,1,n)$ as follows. We first express $v$ along the basis $(\ket{\lambda,\mathbf{s}})$ of $\mathcal{F}_{\mathbf{s}}^n$:
\[
  v = \sum_{\lambda}a_{\lambda}(q)\ket{\lambda,\mathbf{s}},
\]
where $\lambda$ runs in the set of $s$-partitions of $n$ and $a_\lambda(q)\in\mathbb{N}[q,q^{-1}]$. We then obtain the character $\gamma_v$ by evaluation at $q=1$:
\[
  \gamma_v = \sum_{\lambda}a_{\lambda}(1)\chi_\lambda.
\]

\subsection{Canonical basis}
\label{sec:canonical}

Let $x\mapsto \overline{x}$ be the $\mathbb{Q}$-linear algebra involution of $U_q(\mathfrak{sl}_{\infty})$ given by
\begin{align*}
  \overline{q}&=q^{-1}, & \overline{K_i}&= K_i^{-1}, & \overline{E_i}&=E_i, & \overline{F_i} = F_i.
\end{align*}

We consider the submodule $V_{\mathbf{s}}$ of $\mathcal{F}_{\mathbf{s}}$ generated by the highest weight vector $\ket{\emptyset,\mathbf{s}}$, which is then isomorphic to the integrable irreducible module of highest weight $\Lambda = \sum_{i=1}^l\Lambda_{s_i}$. Any element $v\in V_{\mathbf{s}}$ can then be expressed as $v=x\cdot\ket{\emptyset,\mathbf{s}}$ for some $x\in U_q(\mathfrak{sl}_{\infty})$ and we set $\overline{v} = \overline{x}\cdot \ket{\emptyset,\mathbf{s}}$. This gives a well-defined involution on $V_{\mathbf{s}}$.

Let $R$ be the subring of $\mathbb{Q}(q)$ of rational functions regular at $q=0$ and consider $\mathcal{F}_{\mathbf{s},R}$ the $R$-sublattice of $\mathcal{F}_{\mathbf{s}}$ generated by $(\ket{\lambda,\mathbf{s}})_{\lambda}$ for $\lambda$ running in the $l$-partitions of integers. In order to state the existence of canonical basis, we need the definition of a cylindrical multipartition.

\begin{defn}
  Suppose that $\mathbf{s}=(s_1,\ldots,s_l)$ is such that $s_1 \geq s_2 \geq \cdots \geq s_l$. A $l$-partition $\lambda=(\lambda^{(1)},\ldots,\lambda^{(l)})$ is said to be cylindrical if $\lambda^{(k)}_{i+s_{k}-s_{k+1}}\geq \lambda_{i}^{(k+1)}$ for all $1 \leq k \leq l$ and $i \in \mathbb{N}^*$. 
\end{defn}

\begin{thm}[\cite{lusztig-canonical,kashiwara}]
  There exists a unique basis $\left\{G(\lambda,\mathbf{s})\ \middle\vert \ \lambda \text{ cylindrical }l-\text{partition}\right\}$ of $V_{\mathbf{s}}$ such that
  \begin{itemize}
  \item $\overline{G(\lambda,\mathbf{s})} = G(\lambda,\mathbf{s})$,
  \item $G(\lambda,\mathbf{s}) = \ket{\lambda,\mathbf{s}} \mod q\mathcal{F}_{\mathbf{s},R}$.
  \end{itemize}
\end{thm}

The basis $(G(\lambda,\mathbf{s}))$ of $V_{\mathbf{s}}$ is the canonical basis of $V_{\mathbf{s}}$. As for quasimonomial vectors, we can evaluate at $q=1$ and obtain (virtual) characters of $G(l,1,n)$. It is expected that these characters should coincide with the Calogero--Moser cellular characters.

\begin{rem}
The cylindrical multipartitions are the multipartitions which naturally indexed the crystal basis of $V_{\mathbf{s}}$ in the case where $s_1 \geq s_2 \geq \cdots \geq s_l$. 
In the general case where ${\bf s}\in \mathbb{Z}^l$, the  analogous multipartitions  are the images of  the cylindrical multipartitions 
 by certain combinatorial maps described  in \cite{jacon-lecouvey}. 
\end{rem}

\subsection{Another description of the Fock space}
\label{sec:fock_space_2}

We follow the description of the Fock space given in \cite{leclerc-miyachi}. Let $V(\Lambda_k)$ be the irreducible integrable representation of $U_q(\mathfrak{sl}_{\infty})$ of highest weight $\Lambda_k$. A basis of $V(\Lambda_k)$ is given by $(v_\beta)_\beta$  where $\beta$ runs in the set of sequences $(\beta_j)_{j\leq k}$ such that $\beta_{j-1} < \beta_{j}$ for all $j \leq k$ and $\beta_{j} = j$ for $j\ll 0$. The action of $U_q(\mathfrak{sl}_{\infty})$ on $V(\Lambda_k)$ is given by
\begin{align*}
  E_iv_\beta &=
               \begin{cases}
                 v_\gamma & \text{if }i\not\in\beta\text{ and } i+1\in\beta,\text{ with } \gamma=(\beta\setminus\{i+1\})\cup\{i\},\\
                 0 & \text{otherwise},
               \end{cases}
\\
  F_iv_\beta &=
               \begin{cases}
                 v_\gamma & \text{if }i\in\beta\text{ and } i+1\not\in\beta,\text{ with } \gamma=(\beta\setminus\{i\})\cup\{i+1\},\\
                 0 & \text{otherwise},
               \end{cases}\\
  K_iv_\beta &=
               \begin{cases}
                 qv_\beta & \text{if }i\in\beta\text{ and } i+1\not\in\beta,\\
                 q^{-1}v_\beta & \text{if }i\not\in\beta\text{ and } i+1\in\beta,\\
                 v_\beta & \text{otherwise}.
               \end{cases}
\end{align*}

Given $\mathbf{s}=(s_1,\ldots,s_l)\in\mathbb{Z}^l$, the Fock space $\mathcal{F}'_{\mathbf{s}}$ is the tensor product
\[
  \mathcal{F}'_{\mathbf{s}} = V(\Lambda_{s_1})\otimes \cdots \otimes V(\Lambda_{s_l}),
\]
the $U_q(\mathfrak{sl}_{\infty})$-action is given by the coproduct. A basis of $\mathcal{F}'_{\mathbf{s}}$ is given by $(v_S)_S$, with $S$ a $l$-symbol
\[
  S =
  \begin{pmatrix}
    \beta_1\\
    \vdots\\
    \beta_l
  \end{pmatrix},
\]
where $\beta_i = (\beta_{i,j})_{j\leq s_i}$ is a sequence of integers such that $\beta_{i,j-1} < \beta_{i,j}$ for all $j \leq s_i$ and $\beta_{i,j} = j$ for $j\ll 0$. The vector $v_S$ is the tensor product $v_{\beta_1}\otimes \cdots \otimes v_{\beta_l}$.

Such symbols are in bijection with $l$-partitions. Given such a symbol $S$ we associate a $l$-partition $\lambda_S=(\lambda^{(1)},\ldots,\lambda^{(l)})$ by setting
\[
  \lambda^{(k)}_i = \beta_{k,s_k-i+1}-s_k+i-1.
\]
Thanks to the assumptions on the entries of the symbol $S$, we do have $\lambda^{(k)}_i \geq \lambda^{(k)}_{i+1}$ and $\lambda^{(k)}_i = 0$ for $i \gg 0$. It is then easy to see that $S\mapsto \lambda_S$ is a bijective map.

We can then detect the addable and removable nodes of the $l$-partition $\lambda_S$ from the symbol $S$. Indeed, a node $\gamma=(a,b,c)$ is removable if and only if $\beta_{c,s_c-a+1}=b-a+s_c+1$ and $b-a+s_c$ does not appear in $\beta_c$. Similarly, a node $\gamma=(a,b,c)$ is addable to $\lambda$ if and only if $\beta_{c,s_c-a+1}=b-a+s_c$ and $b-a+s_c+1$ does not appear in $\beta_c$. Since $b-a+s_c$ is nothing more than the charged content of $\gamma$, we can also recover the content of a removable or addable node from the symbol $S$.

\begin{ex}
  Consider the multicharge $(3,1,2)$ and the following symbol $S$
  \[
    S=\begin{pmatrix}
      \cdots & -1 & 0 & 2 & 5 & 7\\
      \cdots & -1 & 0 & 4 \\
      \cdots & -1 & 1 & 2 & 3
    \end{pmatrix}.
  \]
  It corresponds to the $3$-partition $\lambda_S=(4.3.1,3,1^3)$ whose Young diagram is
  \[
    \ytableausetup{aligntableaux=top}
    [\lambda_S]=\left(
      \begin{ytableau}
      3 & 4 & 5 & *(gray) 6\\
      2 & 3 & *(gray) 4\\
      *(gray) 1
    \end{ytableau}
    ,
    \begin{ytableau}
      1 & 2 & *(gray) 3
    \end{ytableau}
    ,
    \begin{ytableau}
      2\\
      1\\
      *(gray) 0
    \end{ytableau}
    \right)
  \]
  the entries corresponding to the charged content. Removable nodes are in gray and correspond to the entries $2,5$ and $7$ in the first row of $S$, to the entry $4$ in the second row of $S$ and to the entry $1$ in the last row of $S$.
\end{ex}

\begin{prop}
  The above bijection between $l$-symbols and $l$-partitions induces an isomorphism of $U_q(\mathfrak{sl}_{\infty})$-modules between $\mathcal{F}'_{\mathbf{s}}$ and $\mathcal{F}_{\mathbf{s}}$ given by $v_{S}\mapsto \ket{\lambda_S,\mathbf{s}}$.
\end{prop}


\begin{proof}
  It suffices to check that the isomorphism of vector spaces $v_{S}\mapsto \ket{\lambda_S,\mathbf{s}}$ is $U_q(\mathfrak{sl}_{\infty})$-equivariant, which is easily checked.
\end{proof}

The vector $\ket{\emptyset,\mathbf{s}}$ then corresponds to the vector $v_{S^0}$ where $S^0$ is the symbol such that its $i$-th row is the sequence $(j)_{j\leq s_i}$.

Finally, note that detecting a cylindrical multipartition is easy from its corresponding symbol: the $l$-partition $\lambda_s$ is cylindrical if and only if the symbol $S$ is standard, that is if $\beta_{i,k}\geq \beta_{i,k+1}$ for all $1 \leq k < l$ and $i \leq s_{k+1}$.

\subsection{Comparison with cellular characters}
\label{sec:comparison}

~

\boitegrise{{\bf Hypothesis:} We suppose that the parameter $c$ is such that there exists $k\in \mathbb{C}$ with $lk_i\in \mathbb{Z}+k$ for every $i$ and that $c_0=-\frac{1}{l}$.}{0.8\textwidth}


We set $\mathbf{k'}=(k'_1,\ldots,k'_l)\in\mathbb{Z}^l$ where $k'_{c} = lk_{1-c}-k$. Therefore, given a standard tableau $\mathfrak{t}$ of shape $\lambda$, the shift of the Jucys--Murphy element $J_i-k$ acts on the line $D_{\mathfrak{t}}$ by multiplication by the charged content of $\mathfrak{t}^{-1}(i)$.

\begin{thm}
  \label{thm:JM_monomial}
  The set of Jucys--Murphy cellular characters of $G(l,1,n)$ for the parameter $c$ coincide with the characters obtained by evaluation at $q=1$ of quasimonomial vectors of $\mathcal{F}_{\mathbf{k'}}^n$.
\end{thm}

\begin{proof}
  If $v\in\mathcal{F}_{\mathbf{k'}}^{n-1}$ and $v'=F_i\cdot v \in \mathcal{F}_{\mathbf{k'}}^{n}$ are quasimonomial vectors, then it follows from the definition of the action and the particular choice of parameter that $\gamma_{v'} = i-\Ind_{G(l,1,n-1)}^{G(l,1,n)}(\gamma_v)$. Therefore the characters of $G(l,1,n)$ obtained from specializations at $q=1$ of quasimonomial vectors coincide with the set $\mathcal{E}_n$ introduced in \cref{sec:gl1n}, since the only quasimonomial vector of $\mathcal{F}_{\mathbf{k'}}^0$ is $\ket{\emptyset,\mathbf{k'}}$. Therefore, we conclude using \cref{prop:JM-char-truncated-induction}. 
\end{proof}

Using \cref{prop:JM_sums_of_CM}, we deduce that the characters obtained by evaluation at $q=1$ of quasimonomial vectors of $\mathcal{F}_{\mathbf{k'}}^n$ are sums of Calogero-Moser $c$-cellular characters. The second author has also formulated the following conjecture concerning vectors of the canonical basis of $V_{\mathbf{s}}$:

\begin{conj}
  The set of Calogero--Moser cellular characters of $G(l,1,n)$ for the parameter $c$ coincide with the characters obtained by evaluation at $q=1$ of the vectors of height $n$ of the canonical basis of $V_{\mathbf{s}}$.
\end{conj}



\section{The type $B$ and Lusztig's constructible characters}
\label{sec:typeB}

In this Section, we take $l=2$ so that the complex reflection group $G(l,1,n)$ is isomorphic to the Weyl group of type $B_n$. We aim to compare, using results of Leclerc--Miyachi \cite{leclerc-miyachi}, the Jucys--Murphy cellular characters and Lusztig's constructible characters.

Lusztig's constructible characters depend also on a choice of parameters for each conjugacy class of reflections. We will always take the parameter equal to $1$ for the conjugacy class $S_0$ and the parameter for the conjugacy class $S_1$ is denoted by $r$ and will be explicitely related to the multicharge $\mathbf{s}$.

\begin{thm}[{\cite[Theorem 10 and Proposition 4]{leclerc-miyachi}}]
  \label{thm:LM}
  Let $\mathbf{s}=(s_1,s_2)\in\mathbb{Z}^2$ with $s_1 \geq s_2$. Then the characters of the Weyl group of type $B$ obtained from specializations at $q=1$ of the vectors of the canonical basis of $V_{\mathbf{s}}$ coincide with Lusztig's constructible characters at parameter $r=s_1-s_2$. 

Moreover, every element in the canonical basis of $V_{\mathbf{s}}$ is of the following form:
\[
  G(\lambda,\mathbf{s}) = F_{i_k}^{(r_k)}\cdots F_{i_1}^{(r_1)}\cdot v_{S^0},
\]
where $i_j\in\mathbb{Z}$ and $r_j\in\{1,2\}$ for all $1 \leq j \leq k$.
\end{thm}

Because of the divided powers, it is not clear if the vectors of the canonical basis are monomial or not.

\begin{prop}
  \label{prop:simplify_canonical}
  Let $\mathbf{s}=(s_1,s_2)\in\mathbb{Z}^2$ with $s_1 > s_2$. For any $n\in\mathbb{N}$, any $k \in \mathbb{N}$, any $i_1,\ldots,i_k\in\mathbb{Z}$ and any $r_1,\ldots,r_k\in\{1,2\}$ such that $\sum_{i=1}^kr_i = n$ there exist $j_1,\ldots,j_n$ such that
  \[
    F_{i_k}^{(r_k)}\cdots F_{i_1}^{(r_1)}\cdot v_{S^0} = F_{j_n}\cdots F_{j_1}\cdot v_{S^0}.
  \]
\end{prop}

\begin{proof}
  We proceed by induction on $n$. There is nothing to prove if $n=0$ or if $n=1$. Note that since $s_1>s_2$, we have $F_i^{(2)}\cdot v_{S^0}=0$, which also settles the case $n=2$.

  Now suppose that $n\geq 3$ and that the result is proven for every $l \leq n-1$. Fix $i_1,\ldots,i_k\in\mathbb{Z}$ and $r_1,\ldots,r_k\in\{1,2\}$ such that $\sum_{i=1}^kr_i = n$. Of course, we suppose that $F_{i_{k}}^{(r_{k})}\cdots F_{i_1}^{(r_1)}\cdot v_{S^0}$ is non-zero, otherwise there is nothing to show.

  If $r_k=1$, then the induction hypothesis gives us $j_1,\ldots,j_{n-1}$ such that
  \[
    F_{i_{k-1}}^{(r_{k-1})}\cdots F_{i_1}^{(r_1)}\cdot v_{S^0} = F_{j_{n-1}}\cdots F_{j_1}\cdot v_{S^0},
  \]
  and acting by $F_{i_k}$ shows that
  \[
    F_{i_k}F_{i_{k-1}}^{(r_{k-1})}\cdots F_{i_1}^{(r_1)}\cdot v_{S^0} = F_{i_k}F_{j_{n-1}}\cdots F_{j_1}\cdot v_{S^0}
  \]
  has the desired form.

  Finally, suppose that $r_k=2$. We use the induction hypothesis to obtain $j_1,\ldots,j_{n-2}$ such that
  \[
    F_{i_{k-1}}^{(r_{k-1})}\cdots F_{i_1}^{(r_1)}\cdot v_{S^0} = F_{j_{n-2}}\cdots F_{j_1}\cdot v_{S^0},
  \]
  and we now discuss on the value of $j_{n-2}$. If $j_{n-2}\not\in\{i_k-1,i_k,i_k+1\}$ then $F_{i_k}^{(2)}$ and $F_{j_{n-1}}$ commute and then
  \[
    F_{i_k}^{(2)}F_{i_{k-1}}^{(r_{k-1})}\cdots F_{i_1}^{(r_1)}\cdot v_{S^0} = F_{j_{n-2}}F_{i_k}^{(2)}F_{j_{n-3}}\cdots F_{j_1}\cdot v_{S^0}.
  \]
  We once again apply the induction hypothesis to $F_{i_k}^{(2)}F_{j_{n-3}}\cdots F_{j_1}\cdot v_{S^0}$ to reach the conclusion.

  Since we have supposed that $F_{i_k}^{(2)}F_{i_{k-1}}^{(r_{k-1})}\cdots F_{i_1}^{(r_1)}\cdot v_{S^0}$ is non-zero, we have $j_{n-2}\neq i_k$. Indeed, $F_{i_k}^{3}$ acts by zero on the Fock space by Remark \ref{diffcon}. 

  If $j_{n-2} = i_{k}-1$, we want to use the quantum Serre relation $F_{i_k}^{(2)}F_{j_{n-2}} = F_{i_k}F_{j_{n-2}}F_{i_k} - F_{j_{n-2}}F_{i_k}^{(2)}$ and we show that $F_{i_k}^{(2)}$ acts by zero on $F_{j_{n-3}}\cdots F_{j_1}\cdot v_{S^0}$. If a symbol $S=\binom{\beta}{\gamma}$ is such that $v_S$ appears in $F_{j_{n-3}}\cdots F_{j_1}\cdot v_{S^0}$ with a non zero coefficient, then the value of $\beta\cap \gamma$ and $\beta\cup \gamma$ does not depend on the chosen symbol. Then $i_k \not \in \beta\cap \gamma$, otherwise $F_{i_{k-1}}$ would act by zero on $F_{j_{n-3}}\cdots F_{j_1}\cdot v_{S^0}$. Therefore $F_{i_k}^{(2)}$ acts by zero on $F_{j_{n-3}}\cdots F_{j_1}\cdot v_{S^0}$ and we have
  \[
    F_{i_k}^{(2)}F_{j_{n-2}}F_{j_{n-3}}\cdots F_{j_1}\cdot v_{S^0} = F_{i_k}F_{j_{n-2}}F_{i_k}F_{j_{n-3}}\cdots F_{j_1}\cdot v_{S^0},
  \]
  which has the expected form.

  If $j_{n-2} = i_k+1$, we can proceed similarly: we show that  $i_k+1 \in \beta\cup \gamma$ and then $F_{i_k}^{(2)}$ acts by zero on $F_{j_{n-3}}\cdots F_{j_1}\cdot v_{S^0}$.
\end{proof}

\begin{ex}
  If we take the multicharge $\mathbf{s}=(1,0)$ then the vector $F_{0}^{(2)}F_{1}\cdot v_{S^0}$ is an element of the canonical basis. Since $F_0^{(2)}$ acts by zero on $v_{S^0}$, we have $F_{0}^{(2)}F_{1}\cdot v_{S^0}=F_{0}F_{1}F_{0}\cdot v_{S^0}$, and we obtain a monomial vector.
\end{ex}

If the multicharge $\mathbf{s}$ is of the form $(s,s)$, we have almost the same result.

\begin{prop}
  Let $\mathbf{s}=(s,s)\in\mathbb{Z}^2$. For any $n\in\mathbb{N}$, any $k \in \mathbb{N}$, any $i_1,\ldots,i_k\in\mathbb{Z}$ and any $r_1,\ldots,r_k\in\{1,2\}$ such that $\sum_{i=1}^kr_i = n$ and $(r_1,\ldots,r_k)\neq (2,\ldots,2)$, there exist $j_1,\ldots,j_n$ such that
  \[
    F_{i_k}^{(r_k)}\cdots F_{i_1}^{(r_1)}\cdot v_{S^0} = F_{j_n}\cdots F_{j_1}\cdot v_{S^0}.
  \]

  If $n=2k$ is even and $(r_1,\ldots,r_k)=(2,\ldots,2)$ then the vector $F_{i_k}^{(2)}\cdots F_{i_1}^{(2)}\cdot v_{S^0}$ is either zero or of the form $v_S$ for $S$ a symbol of the form $\binom{\beta}{\beta}$.
\end{prop}

\begin{proof}
  The strategy is the same as in the proof of \cref{prop:simplify_canonical}. We proceed by induction on $n$ and the only difference being the case of a non-zero vector of the form $F_{i_k}F_{i_{k-1}}^{(2)}\cdots F_{i_1}^{(2)} \cdot v_{S^0}$. Indeed, we cannot apply the induction hypothesis to $F_{i_{k-1}}^{(2)}\cdots F_{i_1}^{(2)} \cdot v_{S^0}$.

  As $F_{i_k}F_{i_{k-1}}^{(2)}\cdots F_{i_1}^{(2)} \cdot v_{S^0}$ is non-zero, we have $i_k\neq i_{k-1}$. If $i_k\neq i_{k-1}\pm 1$, then $F_{i_k}$ and $F_{i_{k-1}}^{(2)}$ commute and we rewrite $F_{i_{k-1}}^{(2)}F_{i_k}F_{i_{k-2}}^{(2)}\cdots F_{i_1}^{(2)} \cdot v_{S^0}$ as in the proof of \cref{prop:simplify_canonical}.

  If $i_k = i_{k-1}-1$, as we supposed that $F_{i_k}F_{i_{k-1}}^{(2)}\cdots F_{i_1}^{(2)} \cdot v_{S^0}$, the vector $F_{i_{k-2}}^{(2)}\cdots F_{i_1}^{(2)} \cdot v_{S^0}$ is equal to $v_S$ with $S=\binom{\beta}{\beta}$ with $i_{k}+1=i_{k-1}\in \beta$. Therefore $F_{i_k}^{(2)}\cdot v_S=0$ and
  \[
    F_{i_k}F_{i_{k-1}}^{(2)}\cdots F_{i_1}^{(2)} \cdot v_{S^0} = F_{i_{k-1}}F_{i_k}F_{i_{k-1}}F_{i_{k-2}}^{(2)}\cdots F_{i_1}^{(2)} \cdot v_{S^0},
  \]
  and we once again conclude by induction. The case $i_k=i_{k-1}+1$ is similar.
\end{proof}

The two previous propositions enable us te get rid of the divided powers in almost all cases.

\begin{cor}
  \label{cor:canonical_is_monomial}
  Let $\mathbf{s}=(s_1,s_2)\in\mathbb{Z}^2$ be a multicharge with $s_1\geq s_2$.
  \begin{enumerate}
  \item If $s_1>s_2$ then any vector of the canonical basis of $V_{\mathbf{s}}$ is monomial.
  \item If $s_1=s_2$ and then any vector of the canonical basis of $V_{\mathbf{s}}$ different from $v_S$with $S=\binom{\beta}{\beta}$ is monomial.
  \end{enumerate}
\end{cor}

\begin{thm}
  \label{thm:constructible_CM}
  Lusztig's constructible characters in type $B$ at parameter $r$ are sums of Calogero--Moser cellular characters at parameters $c$ with $c_0 = -1/2$ and $c_1 = -r/2$.
\end{thm}

\begin{proof}
  In order to be consistent with the hypothesis of \cref{sec:comparison}, we take $k=-r/2$ so that $(k'_1,k'_2)=(r,0)$.

  If $\gamma$ is a Lusztig's constructible character, then thanks to \cref{thm:LM}, $\gamma$ is obtained from the evaluation at $q=1$ of a vector of the canonical basis of $V_{\mathbf{s}}$.

  If $r\neq0$ \cref{cor:canonical_is_monomial} ensures that $\gamma$ is obtained from the evaluation at $q=1$ of a monomial vector of the Fock space. Therefore, thanks to \cref{thm:JM_monomial}, $\gamma$ is a Jucys--Murphy constructible character, which is a sum of Calogero--Moser $c$-cellular characters, see \cref{prop:JM_sums_of_CM}.

  If $r=0$, then the same argument shows that any Lusztig's constructible character different from $\chi_{(\lambda,\lambda)}$ is a sum of Calogero--Moser $c$-cellular characters. But the character $\chi_{(\lambda,\lambda)}$ is alone in its Calogero--Moser family. Indeed, the two rows of the $2$-symbol $S_\lambda$ corresponding to the bipartition $(\lambda,\lambda)$ are equal: there exists only one symbol with the same content of $S_\lambda$ and the Calogero--Moser family is determined by the content, see \cite[Theorem 3.9 and 3.13]{Martino}. Hence, using \cite[Proposition 7.7.1 and Theorem 13.5.1]{bonnafe-rouquier}, we deduce that $\chi_{(\lambda,\lambda)}$ is also a Calogero--Moser $c$-cellular character.
\end{proof}


\section{Calogero-Moser families and $b$-invariant}
\label{sec:b-invariant}

\subsection{Calogero-Moser families}

Using the representation theory of restricted rational Cherednik algebras at $t=0$, Gordon \cite{gordon-baby} defines the notion of Calogero--Moser families of characters for a complex reflection group $W$. It consists in a partition of the set of irreducible representations for $W$, which, in the case of Weyl group, conjecturally correspond to the notion of Lusztig family. This conjecture is known to hold in the case where $W$ is a classical Weyl group thanks to a case by case analysis.  
 
Now, for each irreducible representation $E$ of $W$, one can define its fake degree $f_E (q)$ which is a polynomial in one indeterminate $q$. This object is connected with the invariant theory of $W$. The valuation of $f_E (q)$ is denoted by $b_{E}$ and is known as the $b$-invariant of $E$.  In \cite{bonnafe-rouquier}, Bonnaf\'e and Rouquier have proved 
 that in each Calogero-Moser family, there is a unique element with  minimal $b$-invariant. Our aim is to recover this result and to give the explicit form of this element in the case where $W=G(l,1,n)$.  

\subsection{The case $W=G(l,1,n)$}

 We first describe the Calogero-Moser family in this case and give an explicit formula for the $b$-invariant of an irreducible character in terms of symbols. Let $l\in \mathbb{N}_{>0}$ and let $\mathbf{s}=(s_1,\ldots,s_l) \in \mathbb{N}^l$. We will suppose that $s_1\geq s_{2} \geq \cdots \geq s_l$ and consider the set of $l$-symbols $\Symb(\mathbf{s})$ of multicharge $\mathbf{s}$. The symbols considered in this section are slightly diffrent from the ones introduces in \cref{sec:monomial-bases}, and we stress the difference of conventions. First, the multicharge is allowed to have only non-negative values. Then, the $i$-th row of an $l$-symbol will be a finite increasing sequence of $s_i$ integers. Infinite symbols as in \cref{sec:monomial-bases} have the advantage to be in bijection with all $l$-partitions of all integers, whereas there are some restriction for finite symbols: we can only recover partitions of integers with the $i$-th component of length at most $s_i$. Therefore, the irreducible representations of $G(l,1,n)$ may be parametrized by $\Symb(\mathbf{s})$ as soon as $s_l \geq n$. In the following, we will need this assuption on $s_l$ every time we want to describe an irreducible representation by its associated symbol. 
We now write a symbol $S\in \Symb(\mathbf{s})$ as follows: 
\[
  S=
  \begin{pmatrix}
    \beta_{1}^1  &   \ldots &    \beta_{s_1-s_2}^1   & \beta_{s_1-s_2+1}^1 & \ldots &  \beta_{s_1-s_l}^1  & \beta_{s_1-s_l+1}^1 &  \ldots & \beta_{s_1}^1  \\
    & & & \beta_{1}^{2}  &  \ldots & \beta_{s_2-s_l}^2   & \beta_{s_2-s_l+1}^2 & \ldots &  \beta_{s_2}^{1}\\
    & & &  &\ldots &\ldots & \ldots & \ldots & \ldots  \\
    & & & & & &  \beta_{1}^l &\ldots   & \beta_{s_l}^l
  \end{pmatrix}.
\]

Given $S\in \Symb(\mathbf{s})$, we define $b(S)\in\mathbb{N}$  and $b'(S)\in\mathbb{N}$ by
\[
  b(S) = \sum_{i=1}^n\sum_{j=1}^{s_i}(l(s_i-j)+i-1)(\beta^i_j-j+1) \quad\text{and}\quad
  b'(S) = \sum_{i=1}^n\sum_{j=1}^{s_i}(l(s_i-j)+i-1)\beta^i_j.
\]
Then $b(S)$ is the $b$-invariant associated with the simple representation parametrized by $S$. This can be deduced from the computation of the fake degrees in \cite[Bemerkung 2.10]{malle-unipotente}. It is clear that for two symbols $S$ and $S'$ associated  with multipartitions with the same length, we have:
\[
b(S)\leq b(S') \iff b'(S)\leq b'(S').
\]
Given a symbol $S\in\Symb(\mathbf{s})$, we associate a sequence $z(S)$ by
\begin{multline*}
  z(S) = (\beta_1^1,\ldots,\beta_{s1-s2}^1,\beta^2_1,\beta^1_{s1-s2+1},\ldots,\beta^2_{s2-s3},\beta^1_{s1-s3},\beta^3_1,\beta^2_{s2-s3+1},\beta^1_{s1-s3+1},\ldots,\\
  \beta^l_1,\beta^{l-1}_1,\ldots,\beta^1_1).
\end{multline*}
This sequence is obtained from the symbol $S$ by reading the entries from bottom to top and from left to right.

Let $W$ be the tuple $(w^1,\ldots,w^l)$ such that $b'(S)=\sum_{i=1}^{s_1+\cdots+s_l}z(W)_iz(S)_i$, that is
\[
  W =
  \begin{pmatrix}
    l(s_1-1) & \ldots & ls_2 & l(s_2-1) & \ldots & ls_l & l(s_l-1) & \ldots & 0\\
    &        &      & l(s_2-1)+1 & \ldots & ls_l+1 & l(s_l-1)+1 & \ldots & 1\\
    &        &      &            & \ldots & \ldots & \ldots & \ldots & \ldots\\
    &   &  &  &  &  & l(s_l-1)+l-1 & \ldots & l-1
  \end{pmatrix}
\]

Note that $z(W)$ is a strictly decreasing sequence (eventhough $W$ in not a symbol, the definition of $z(W)$ still makes sense).

Then the  two symbols $S$ and $T$ in $\Symb(\mathbf{s})$ belong to the same Calogero--Moser family if the underlying multiset is the same; this defines a partition of $\Symb(\mathbf{s})$ into families of symbols.

Fix $\mathcal{F}$ a family of symbols with underlying multiset of entries $\mathcal{E}$. We define a symbol $S_{\mathcal{F}}\in \mathcal{F}$ through the sequence $z(S_{\mathcal{F}})$:
\begin{enumerate}
\item we set $z(S_{\mathcal{F}})_1$ to be the minimum of $\mathcal{E}$,
\item suppose that we have defined $z(S_{\mathcal{F}})_1,\ldots,z(S_{\mathcal{F}})_i$; then we set $z(S_{\mathcal{F}})_{i+1}$ to be the minimal element $x$ of $\mathcal{E}\setminus\{z(S_{\mathcal{F}})_1,\ldots,z(S_{\mathcal{F}})_i\}$ such that there exists a symbol $S\in \mathcal{F}$ with associated sequence $z(S)$ starting with $(z(S_{\mathcal{F}})_1,\ldots,z(S_{\mathcal{F}})_i,x)$.
\end{enumerate}

\begin{ex}
  Here, we take the multicharge $(5,5,3)$ and the family with underlying multiset $\{0,0,0,1,1,2,2,5,7,8,9,11,12\}$. Then we have $z(S_{\mathcal{F}})=(0,0,1,1,0,2,2,5,7,8,9,11,12)$ and the corresponding symbol is 
  \[
    S_{\mathcal{F}}=
    \begin{pmatrix}
      0 & 1 & 2 & 8 & 12\\
      0 & 1 & 2 & 7 & 11\\
      &   & 0 & 5 & 9\\
    \end{pmatrix}.
  \]
\end{ex}

\begin{thm}
  \label{thm:minimal_binv}
  Let $\mathcal{F}\subset \Symb(\mathbf{s})$ be a family of symbols. Then $b(S_{\mathcal{F}})\leq b(S)$ for all $S\in \mathcal{F}$ with equality if and only if $S=S_{\mathcal{F}}$.
\end{thm}

Let $\unlhd$ be the order on $\Symb(\mathbf{s})$ induced by the dominance order on the associated sequences:
\[
  S \unlhd T \text{ if and only if } \forall 1 \leq j \leq s_1+\cdots+s_l, \sum_{i=1}^jz(S) \leq \sum_{i=1}^jz(T).
\]

From the definition of $b'$, it is easy to see that
\[
  b'(S) = \sum_{i=1}^{s_1+\cdots+s_l-1}(z(W)_i-z(W)_{i+1})\sum_{j=1}^iz(S)_j.
\]
Therefore if $S \unlhd T$ then $b'(S)\leq b'(T)$ and if $S\unlhd T$ and $b'(S) = b'(T)$ then $S=T$. Hence Theorem \ref{thm:minimal_binv} will follow immediately from the next proposition.

\begin{prop}
  Let $\mathcal{F}$ be a family of $\Symb(\mathbf{s})$. Then $S_{\mathcal{F}} \unlhd S$ for all $S\in \mathcal{F}$.
\end{prop}
\begin{proof}
  Let $S\in\mathcal{F}$ different from $S_{\mathcal{F}}$ and consider $1 \leq j \leq s_1+\cdots+s_l$ minimal such that $z(S)_j\neq z(S_{\mathcal{F}})_j$. We construct a symbol $T$ such that $T\unlhd S$ and $z(T)_i = z(S_{\mathcal{F}})_i$ for all $1 \leq i \leq j$. By induction, this implies that $S_{\mathcal{F}}\unlhd T \unlhd S$.

  Let $j < k \leq s_1+\ldots + s_l$ be minimal such that $z(S)_k = z(S_{\mathcal{F}})_j$. We consider the two rows $S^a$ and $S^b$ of $S$ such that $\beta^a_p$ is the entry $z(S)_j$ and $\beta^b_q$ the entry $z(S)_k$. Note that, by definition of $S_{\mathcal{F}}$, we must have $z(S)_j>z(S_{\mathcal{F}})_j$ and $z(S_{\mathcal{F}})_j$ does not appear in $S^a$.

  Since $k>j$, we have $s_b-q+1 \leq s_a-p+1$. Therefore there exists $p' \geq p$ such that $\beta^a_{p'}\neq \beta^b_{r}$ for all $r\geq q$. Indeed, if $\beta^a_{p'}$ belongs to $S^b$ for all $p'\geq p$, then these $s_a-p+1$ different integers all bigger that $z(S)_k=\beta^b_q$ appear at the right of $\beta^p_q$. But there are $s_b-q < s_a-p+1$ integers to the right of $\beta^b_q$.

  Now that we have find an entry in $S^a$ at the right of $z(S)_j$ ($z(S)_j$ included) which does not appear in $S^b$, we now define the symbol $T$ from $S$. Denote by $q_0=q,q_1,\ldots,q_{p'-p-1}$ the integers such that $\beta^a_{p+r} = \beta^b_{q_{r+1}}$ for all $0 \leq r \leq p'-p-1$. 
   Note that $\beta_{q_0}^b=\beta_q^b$ cannot appear in $S^a$ otherwise, it would appear at the left of $\beta_p^a$ and thus it would appear with multiplicity $2$ in row $S_{\mathcal{F}}^a$.

  Let $T$ be the symbol obtained by exchanging $\beta^a_{p+r}$ with $\beta^b_{q_r}$ for $0 \leq r \leq p'-p$ and increasingly reordering the row in position $b$. Since $z(T)$ is obtained from $z(S)$ by a successive exchange of pairs $(x,y)$ with $x>y$ and $x$ appearing before $y$, we have $T\unlhd S$. Moreover, by construction, the $j-1$ first entries of $z(T)$ and $z(S)$ coincide with those of $z(S_{\mathcal{F}})$ and $z(T)_j = z(S_{\mathcal{F}})_j$.  
\end{proof}



\section{Conjectures}
\label{sec:conjectures}

\subsection{Monomial and canonical bases} In section \ref{sec:typeB}, we have seen that every canonical basis elements are in fact monomial vectors in the case where $l=2$. 

We conjecture that this is also the case when $l=3$ even if a closed formula for them might been more difficult to obtain than in the case $l=2$. When $l>2$, the canonical basis elements are not monomial in general, but such counter-examples are not easy to find.

From explicit computer calculations\footnote{A program for the computation of monomial and canonical basis elements is available in \url{http://ftp.math.rwth-aachen.de/homes/CHEVIE/contr/arikidec.g}}, we found the following counter-example for $l=5$ with the multicharge $\mathbf{s}=(3,2,2,1,0)$ and the multipartition $\lambda=((3),(3),(1),\emptyset,\emptyset)$. The corresponding element of the canonical basis $G(\lambda,\mathbf{s})$ is not a monomial vector of the Fock space. 
  
\subsection{Canonical basis and $b$-invariant}

In \cite[Theorem 11.4.2(c)]{bonnafe-rouquier}, Bonnaf\'e and Rouquier have proved that every irreducible constituent of a Calogero--Moser $c$-cellular character belongs to the same Calogero--Moser family. Therefore, one can define the Calogero--Moser family associated to a Calogero--Moser $c$-cellular character. Note that vectors of the standard basis appearing in the expression of the a vector of the canonical basis have a similar behaviour.

As for Calogero--Moser families, Bonnaf\'e and Rouquier have proven a result in relation with the $b$-invariant: every Calogero--Moser $c$-cellular characters have a unique irreducible constituent of minimal $b$-invariant. This result can be thought as an extension of a similar property for Lusztig's constructible characters for Coxeter groups, see \cite{bonnafe-binv}.

Thanks to computer calculations, we conjecture some similar properties for the vectors of the canonical basis.

\begin{conj}
  We fix $\mathbf{s}=(s_1,\ldots,s_l) \in \mathbb{Z}^l$ and consider $\lambda$ a cylindric $l$-partition. We expand $G(\lambda,\mathbf{s})$ along the standard basis:
  \[
    G(\lambda,\mathbf{s}) = \sum_{\mu}a_{\lambda,\mu}(q)\ket{\mu,\mathbf{s}},
  \]
with $a_{\lambda,\mu}(q) \in \mathbb{Z}[q,q^{-1}]$.
  
  Then there exists a unique $\nu$ with minimal $b$-invariant among the $\mu$ such that $a_{\lambda,\mu}(q)\neq 0$.
\end{conj}

If one considers Lusztig's constructible characters at equal parameters, given a Lusztig's family of characters $\mathcal{F}$, every constructible character whose underlying family $\mathcal{F}$ share an irreducible component, which is the character of minimal $b$-invariant in the family $\mathcal{F}$. We also conjecture a similar behavior for the canonical basis, for specific parameters $\mathbf{s}$. Let us denote by $m(\lambda)$ the $l$-partition $\nu$ of the previous conjecture. 

\begin{conj}
  Suppose that the parameter $\mathbf{s}$ is of the following form: $\mathbf{s} = (s+1,\ldots,s+1,s,\ldots,s)$. Let $\lambda$ be a cylindrical $l$-partition. Then $m(\lambda)$ is the symbol of minimal content in the Calogero--Moser family containing $\lambda$.
\end{conj}

This conjecture  asserts that the constituent with minimal $b$-invariant is the same if the two vectors of the canonical basis correspond to the same family of symbols.

We finally give a last conjecture,  which corresponds to \cite[Lemma 4.6]{Geck} and \cite[Remarque 5.7]{malle-rouquier}. This conjecture is related with the so-called spetsial parameters. These parameters should be the correct generalization of the equal parameters for Hecke algebras associated with equal parameters, see \cite[Chapter 8]{bonnafe-rouquier} for the correspondance of parameters between rational Cherednik algebras and cyclotomic Hecke algebras. In terms of the multicharge $\mathbf{s}$, these spetsial parameters correspond to the case $\mathbf{s}=(s+1,s,\ldots,s)$. Note that symbols of shape given by $\mathbf{s}$ are considered in \cite{malle-unipotente}, where they parametrize the unipotent degrees of the principal family for the imprimitive complex reflection group $G(l,1,n)$.

Let $s_E$ be the Schur element associated with the representation $E$ of the Ariki--Koike algebra at spetsial parameters. These elements have been explicitly computed, see \cite[Theorem 2.4]{chlouveraki} for a factorization of these Schur elements. The leading coefficient of $s_E$ is denoted by $\xi_E$. Since the Ariki--Koike algebra is a flat deformation of the group algebra of $G(l,1,n)$, the irreducible representations are parametrized by $l$-partitions of $n$, and we write $\xi_{\lambda}$ instead of $\xi_E$ if $E$ is the representation indexed by $\lambda$.

\begin{conj}
  Let $\mathbf{s}$ be $(s+1,s,\ldots,s)$ the spetsial parameters. Let $\lambda$ be a cylindrical $l$-partition. Then
  \[
    a_{\lambda,m(\lambda)}(1) = \sum_{\mu}\frac{a_{\lambda,\mu}(1)}{\xi_{\mu}}.
  \]
\end{conj}



\bibliographystyle{bibliography/habbrv}
\bibliography{bibliography/cm-sum-constructibles}


\end{document}